\newif\ifpictures
\picturestrue

\documentclass[12pt]{amsart}
\usepackage{amssymb}
\usepackage{amsmath}
\usepackage{amscd}
\usepackage[pdftex]{graphicx}
\usepackage[normalem]{ulem}
\usepackage{hyperref}
\usepackage{bbm} %

\numberwithin{equation}{section}
\newtheorem{thm}{Theorem}
\newtheorem{prop}[thm]{Proposition}
\newtheorem{lemma}[thm]{Lemma}
\newtheorem{cor}[thm]{Corollary}

\newtheorem{fact}[thm]{Fact}

\theoremstyle{definition}

\newtheorem{example}[thm]{Example}
\newtheorem{remark1}[thm]{Remark}
\newtheorem{openproblem1}[thm]{Open problem}
\newtheorem{definition}[thm]{Definition}

\numberwithin{thm}{section}

\newcounter{FNC}[page]
\def\newfootnote#1{{\addtocounter{FNC}{2}$^\fnsymbol{FNC}$%
     \let\thefootnote\relax\footnotetext{$^\fnsymbol{FNC}$#1}}}

\newcommand{\C}{\mathbb{C}}
\newcommand{\N}{\mathbb{N}}

\newcommand{\R}{\mathbb{R}}

\newcommand{\compl}{\mathsf{c}}

\DeclareMathOperator{\conv}{conv}
\DeclareMathOperator{\cone}{cone}

\DeclareMathOperator{\inter}{int}

\DeclareMathOperator{\diag}{diag}

\DeclareMathOperator{\tr}{tr}
\DeclareMathOperator{\cl}{cl}

\let\Im\relax
\DeclareMathOperator{\Im}{Im}

\newcommand{\sym}{\mathcal{S}}

\usepackage{xcolor}

\title[Conic stability of polynomials]{Conic stability of polynomials}
\author{Thorsten J\"orgens}
\author{Thorsten Theobald}

\address{Goethe-Universit\"at, FB 12 -- Institut f\"ur Mathematik,
Postfach 11 19 32, D--60054 Frankfurt am Main, Germany}

\email{\{joergens,theobald\}@math.uni-frankfurt.de}

\date{\today}

\begin{document}
\begin{abstract}
  We introduce and study the notion of conic stability of 
  multivariate complex polynomials in $\mathbb{C}[\mathbf{z}]$, which 
  naturally generalizes the stability of multivariate
  polynomials.
  In particular, we generalize Borcea's and Br\"and\'en's 
  multivariate version of the Hermite-Kakeya-Obreschkoff Theorem 
  to the conic stability and provide a characterization in terms 
  of a directional Wronskian.
  And we generalize a major criterion for stability of determinantal polynomials
  to stability with respect to the positive semidefinite cone.    
\end{abstract}

\maketitle

\section{Introduction}

Stable polynomials have a rich history (see, e.g., 
\cite{rahman-schmeisser-2002}) and attracted a lot of interest 
in recent years. Prominent
research directions include the generalization of classical results on
univariate stable polynomials to multivariate stable polynomials (see, e.g.,
\cite{borcea-braenden-2008, borcea-braenden-leeyang1,
  borcea-braenden-2010, kpv-2015, wagner-2011}) as well
as applications of stable polynomials to various areas of mathematics 
and theoretical
computer science, see
\cite{anari-gharan-2017, borcea-braenden-leeyang2,gurvits-2008,mss-interlacing1,
	mss-interlacing2,pemantle-2012,straszak-vishnoi-2017} and the references therein.
A polynomial $f = f(\mathbf{z}) = f(z_1, \ldots, z_n) \in \C[\mathbf{z}] = \C[z_1,\ldots,z_n]$ is called \emph{stable} if every root $\mathbf{z}$ satisfies $\Im(z_j) \leq 0$ for some $j$.
A stable polynomial $f$ with real coefficients is called \emph{real stable}.

In \cite{TTT}, the authors and de Wolff introduced a geometric approach to stability phenomena introducing the imaginary projection of a polynomial as the set
\begin{equation}
  \label{eq:im1}
	\mathcal{I}(f) \ = \ \{\Im(\mathbf{z})=(\Im(z_1),\ldots,\Im(z_n)):f(\mathbf{z})=0\} \, ,
\end{equation}
where $\Im(\cdot)$ denotes the imaginary part of a complex number.
Using this notion, stability of $f$ is equivalent to $\mathcal{I}(f)\cap(\R_{>0})^n=\emptyset$. 

This geometric view upon stability of polynomials naturally suggests to 
extend the results of the (usual) stability notion to more general real cones.
In this article, given a cone $K \subset \R^n$, 
a polynomial $f$ is called \emph{$K$-stable} if 
$\mathcal{I}(f)\cap \inter K=\emptyset$, where $\inter K$ denotes the interior of $K$. 
Note that $(\R_{\ge 0})^n$-stability coincides with the usual stability.
And note that setting $\Omega = \R^n + iK$, our $K$-stability
also falls into the more general
class of stability notions which forbid zeroes in an arbitrarily given
complex set $\Omega \subset \C^n$ -- however, as pointed out 
in \cite[p.~81]{wagner-2011}, little can be said on a class of that generality.
For polynomials with matrix variables, we consider the special case where 
$K = \sym_n^{+}$ is the cone of positive semidefinite 
matrices.

In the paper, we initiate to develop a theory of $K$-stability of multivariate polynomials.
To begin with, 
we extend the well-known characterization of stable polynomials in terms
of hyperbolic polynomials to the conic case, see 
Lemma~\ref{lemma:K-stability-univariate}.
Our main contribution is the generalization of
three core results on multivariate stable polynomials
to the conic stability. Firstly, we show that the classical Theorem of Hermite-Kakeya-Obreschkoff,
which has been generalized from the univariate to the multivariate case by 
Borcea and Br\"and\'{e}n \cite[Theorem~6.3.8]{borcea-braenden-2010}, 
can be further generalized to $K$-stability for multivariate polynomials;
see Theorem~\ref{thm:HKO-for-K-stability}.
Secondly, we characterize conic stability with respect to polyhedral
cones and non-polyhedral cones in terms of a directional Wronskian;
see Theorem~\ref{th:polynonpoly}.
Thirdly, we show that Borcea's and Br\"and\'{e}n's 
prominent criterion for stability of 
determinantal polynomials in \cite[Theorem 2.4]{borcea-braenden-2008} can be generalized
to stability with respect to the positive semidefinite cone $\sym_n^+$; see 
Theorem~\ref{thm:psd-stability-determinantal-polynomial}.

Our statements and their proofs apply conic duality, and the generalization of the
stability criterion for determinantal polynomials is given in terms of the 
Khatri-Rao product of matrices.

While our work was mainly motivated by the intrinsic relevance and the structure
of stable polynomials, we note that the case $K = \mathcal{S}_n^{+}$ is 
naturally related
to the Siegel upper half-spaces in the theory of modular forms, see 
Section~\ref{se:siegel}.

Beside the actual statements themselves, we think that these extensions pinpoint
that the conic stability offers a very natural generalized framework for studying
stability issues of multivariate polynomials.

The paper is structured as follows. In Section~\ref{se:prelim}, we collect some known
statements and the connection to Siegel upper half-spaces in the theory
of modular forms. In Section~\ref{se:kstab}, we provide some basic results on
$K$-stable polynomials. Then, in Section~\ref{se:hb-and-hko}, we generalize
the multivariate Hermite-Kakeya-Obreschkoff Theorem to the conic 
setting and study conic stability by means of the directional Wronskian.
 Section~\ref{se:psdstab} contains the generalization of the 
characterization of stability for determinantal polynomials.

\section{Preliminaries\label{se:prelim}}

Throughout the text, bold letters denote $n$-dimensional vectors 
unless noted otherwise.

\subsection{Stability theorems}
As general references on stable polynomials, we refer to 
\cite{borcea-braenden-2010,rahman-schmeisser-2002, wagner-2011}. Note that in
our definition of stability, the zero polynomial is not stable,
in consistency with the convention in~\cite{borcea-braenden-2010}.

For univariate, real stable polynomials $f,g\in\R[z]$, let 
$W(f,g) = f'g-g'f$ denote the \emph{Wronskian} of $f$ and $g$
and write $f\ll g$ if $W(f,g) \leq0$ on $\R$. 
Note that univariate, real stable polynomials are real-rooted. 
In the context of univariate stable polynomials, the following concept of 
interlacing roots naturally appears.
\begin{definition}
	Let $f,g\in\R[z]$ be two univariate, real-rooted polynomials with roots
	$\alpha_1 \leq \alpha_2 \leq \cdots \leq \alpha_{\deg f}$ and $\beta_1\leq \beta_2\leq \cdots\leq \beta_{\deg g}$. We say that
	 $f$ and $g$ \emph{interlace} if their roots alternate, i.e., 
	 $\alpha_1 \leq \beta_1 \leq \alpha_2 \leq \beta_2 \leq \cdots$ 
	 or $\beta_1 \leq \alpha_1 \leq \beta_2 \leq \alpha_2 \leq \cdots$. 
	 If all inequalities are strict, $f$ and $g$ \emph{interlace strictly}.
	
	We say that $f$ \emph{interlaces $g$ properly} 
	(or: $f$ is a \emph{proper interlacing} of $g$), if
	\begin{itemize}
	\item $\beta_{\deg g}\geq \alpha_{\deg f} \geq \beta_{\deg g-1}\geq \alpha_{\deg f-1}\geq\cdots$, when the leading coefficients of $f$ and $g$ have the same sign,
	\item $\alpha_{\deg f}\geq \beta_{\deg g}\geq \alpha_{\deg f-1}\geq\beta_{\deg g-1}\geq\cdots$, when the leading coefficients of $f$ and $g$ have different signs.
	\end{itemize}
\end{definition}
For interlacing polynomials $f$ and $g$, the degrees of $f$ and $g$ can only 
differ by at most~$1$. 
We collect two
classical theorems on univariate stable polynomials 
(see \cite{rahman-schmeisser-2002,wagner-2011}).

\begin{prop}[Hermite-Biehler]\label{thm:HB}
	For $f,g\in\R[z] \setminus \{0\}$, the following are equivalent:
	\begin{enumerate}
	\item $g+if$ is stable.
	\item $f,g$ are real stable and $f\ll g$.
	\item $f,g$ are real stable and $f$ interlaces $g$ properly.
	\end{enumerate}
\end{prop}

Extending the definition of $\ll$ and of interlacing to arbitrary 
$f,g \in \R[x]$ 
by requiring real stability of $f$ and $g$, then condition (2) can be written shortly as
$f \ll g$
and (3) can be written shortly as: $f$ interlaces $g$ properly.

\begin{prop}[Hermite-Kakeya-Obreschkoff (HKO, for short)]\label{thm:HKO-univariate}
	Let $f,g\in\R[z]$. Then $\lambda f+\mu g$ is stable or the zero polynomial for 
    all $\lambda,\mu\in \R$ if and only if $f$ and $g$ interlace or $f \equiv g \equiv 0$.
\end{prop}

Moreover, the following theorem will be useful.

\begin{prop}[Hurwitz, see Theorem 1.3.8 in \cite{rahman-schmeisser-2002}]\label{thm:Hurwitz}
	Let $\{f_j\}_{j\in\N}\subset\C[\mathbf{z}]$ be a sequence of polynomials non-vanishing in a
	connected 
	open set $U\subset\C^n$, and assume it converges to a function $f$ uniformly on compact subsets of $U$. 
	Then $f$ is either non-vanishing on $U$ or it is identically $0$.
\end{prop}

For multivariate polynomials $f,g \in \R[\mathbf{z}]$, one writes $f \ll g$ if $g+if$ is stable
(see, e.g., \cite{borcea-braenden-2010, wagner-2011}; and
note that this makes the multivariate Hermite-Biehler statement a definition rather than
a theorem).
The multivariate version of the HKO Theorem
then has the same format as the univariate version. 
The multivariate theorem was shown in \cite[Theorem~1.6]{borcea-braenden-2010}, see also
\cite[Theorem~2.9]{borcea-braenden-leeyang1}, \cite[Theorem 2.9]{wagner-2011}.

\begin{prop}[Multivariate HKO of Borcea and Br\"and\'{e}n]\label{thm:HKO}
	Let $f,g\in\R[\mathbf{z}]$. Then $\lambda f+\mu g$ is stable or the zero polynomial for 
  all $\lambda,\mu\in \R$ if and only if
  $f\ll g$ or $g\ll f$ or $f \equiv g \equiv 0$.
\end{prop}

An important class of stable polynomials comes from determinantal representations
(\cite[Theorem~2.4]{borcea-braenden-2008}, see also \cite{braenden-2011,HV2007,LPR2005}).

\begin{prop}[Borcea, Br\"and\'{e}n]\label{pr:detcrit}
Let $A_1, \ldots, A_n$ be positive semidefinite $d \times d$-matrices
and $B$ be a Hermitian $d \times d$-matrix, then
\[
  f(\mathbf{x}) \ = \ \det \big( \sum_{j=1}^n x_j A_j + B \big)
\]
is real stable or the zero polynomial.
\end{prop}

Determinantal representations of this kind are relevant in the
context of the Lax conjecture (proven by 
Lewis, Parrilo, Ramana \cite{LPR2005}, see  
\cite[Corollary 6.7]{borcea-braenden-2010} for a formulation on 
stable polynomials) as well as its variations and 
generalizations (see, for example, \cite[Section~5]{borcea-braenden-2008}).

\subsection{Imaginary projections and hyperbolicity cones\label{se:imag-proj}}

Set $A^{\compl} = \R^n \setminus A$ for the complement of
a set $A \subseteq \R^n$, and write $\cl A$ for its closure.
For a polynomial $f \in \C[\mathbf{z}]$, the complement $(\cl \mathcal{I}(f))^{\compl}$
of the imaginary projection~\eqref{eq:im1} consists of finitely many convex components
\cite{TTT}.
In the special case of a non-constant polynomial
$f = \sum_{i=1}^n a_i z_i + a_0$ with real coefficients 
$a_1, \ldots, a_n \in \R$ and $a_0 \in \C$, we
have $\mathcal{I}(f) = \{\mathbf{y} \in \R^n \, : \, \sum_{i=1}^n a_i y_i + \Im(a_0) = 0\}$. 

A homogeneous polynomial $f\in\C[\mathbf{z}]$ is called \emph{hyperbolic} in
a real direction $\mathbf{e}\in\R^n$ if $f(\mathbf{e})\neq0$ and for every $\mathbf{x} \in \R^n$ the function $t\mapsto f(\mathbf{x}+t\mathbf{e})$ has only real roots.
Denote by $C(\mathbf{e})=\{\mathbf{x}\in\R^n:f(\mathbf{x}+t\mathbf{e})=0 \Rightarrow t<0\}$ the \emph{hyperbolicity cone} of $f$ with respect to $\mathbf{e}$.
Then $C(\mathbf{e})$ is convex, $f$ is hyperbolic with respect
to every point $\mathbf{e}'$ in its hyperbolicity cone and $C(\mathbf{e})=C(\mathbf{e}')$ (see \cite{garding-59}). 

For a homogeneous polynomial $f\in\C[\mathbf{z}]$,
the hyperbolicity cones of $f$
coincide with the components of $\mathcal{I}(f)^{\compl}$ \cite{joergens-theobald-2017}.

\subsection{The Siegel upper half-space\label{se:siegel}}

The \emph{Siegel upper half-space} (or \emph{Siegel upper half-plane}) $\mathcal{H}_g$
of  degree $g$ (or genus $g$) is defined as
\[
  \mathcal{H}_g \ =  \{ A \in \C^{g \times g} \text{ symmetric } \, : \, \Im(A) \text{ is positive definite} 
   \} \, ,
\]
where $\Im(A) = (\Im(a_{ij}))_{g \times g}$
(\cite{siegel-1939}, see also, e.g., \cite[\S{2}]{van-der-geer-2008}). 
The Siegel upper half-space constitutes the domain
on which the Siegel theta functions are defined. 
It can be used to parameterize polarized
varieties (see also, for example, \cite[Vol.~1, \S{3.I}]{ito-encyclopedic-1987}
and for the use in elliptic curve cryptography \cite[\S{5.1}]{frey-lange-2006}).

\section{\texorpdfstring{$K$}{K}-stability and psd-stability\label{se:kstab}}

Let $K$ be a \emph{proper} cone in $\R^n$, that is, a full-dimensional, closed and pointed convex cone in $\R^n$.
We consider the following generalization of stability.
Let $\sym_{n}$ be the set of real symmetric $n\times n$-matrices,
and let $\sym_n^+$ and $\sym_n^{++}$ denote its subsets
of positive semidefinite and positive definite matrices.

\begin{definition}\label{de:kstable}
	A polynomial $f \in \C[\mathbf{z}]$ is called \emph{$K$-stable}, if $f(\mathbf{z})\neq0$ whenever $\Im(\mathbf{z})\in \inter K$.
	
  If $f \in \C[Z]$ on the symmetric matrix variables
  $Z = (z_{ij})_{n \times n}$ is $\sym_n^+$-stable, then 
  $f$ is called \emph{positive semidefinite-stable} (for short, \emph{psd-stable}).
\end{definition}

Equivalently, a polynomial $f\in \C[Z]$ on the symmetric matrix variables 
$Z = (z_{ij})_{n \times n}$ is psd-stable if there does not
exist a matrix $Z$ in the Siegel upper half-space $\mathcal{H}_n$ with $f(Z) = 0$.
Note that psd-stability generalizes the usual stability in the sense that 
a polynomial $f(z_1,\ldots,z_n)$ is stable 
if and only if $f(\diag(z_1,\ldots,z_n))$ is psd-stable.

\begin{example} \label{ex:ex1}
	(i) Let $f \in \R[\mathbf{z}]$ be given by $f(\mathbf{z})=\mathbf{a}^T \mathbf{z}+b$, 
where $\mathbf{a}$ is a real $n$-dimensional vector and $b\in\R$.
Then $f$ is $K$-stable if and only if 
$\mathbf{a}\in \inter K^\ast$ or $-\mathbf{a} \in \inter K^{\ast}$, where 
$K^{\ast} = \{\mathbf{y} \in \R^n \, : \, 
  \langle \mathbf {x}, \mathbf{y} \rangle \ge 0 \text{ for all } \mathbf{x} \in K\}$ 
denotes the dual cone 
of $K$ and $\langle \cdot, \cdot \rangle$ is the Euclidean dot product on $\R^n$.

Namely, if $\mathbf{a} \in \inter K^{\ast}$ or $-\mathbf{a} \in \inter K^{\ast}$, 
say, $\mathbf{a} \in \inter K^{\ast}$, then for any $\mathbf{z} = \mathbf{x} + i \mathbf{y} \in \C^n$ with $\mathbf{y} \in \inter K$ we have
\[
  f(\mathbf{z}) \ = \ \langle \mathbf{a}, \mathbf{x} \rangle + i \langle \mathbf{a}, \mathbf{y} \rangle + b, \, 
\]
because $\mathbf{a}$ is real. Hence $\Im f(\mathbf{z}) \neq 0$, and thus $f$ is $K$-stable.

Conversely, let $f$ be $K$-stable. Assuming $\mathbf{a} \not\in \pm \inter K^{\ast}$, there exists 
$\mathbf{y}' \in \inter K$ with $\langle \mathbf{a}, \mathbf{y'} \rangle \le 0$ and 
$\mathbf{y}'' \in \inter K$ with $\langle \mathbf{a}, \mathbf{y}'' \rangle \ge 0$. Hence, there exists
some $\mathbf{y} \in \inter K$ with $\langle \mathbf{a}, \mathbf{y} \rangle = 0$. Choosing
$\mathbf{x} \in \R^n$ with $\langle \mathbf{a},\mathbf{x} \rangle + b = 0$ gives a contradiction to the stability 
of $f$.

	For usual stability, this implies the well-known statement that $f=\mathbf{a}^T\mathbf{z}+b$ is stable if and only if $\mathbf{a}\in(\R_{>0})^n$ or $- \mathbf{a}\in(\R_{>0})^n$.
 For psd-stability, this implies that $f(Z)=\left<Z,A\right>+b$ with
 $A \in \sym_n$ is psd-stable if and only if $A\succ 0$ or $-A \succ 0$; 
 here, the scalar product is $\langle Z, A \rangle = \tr(A^H Z) = \tr(A Z)$
 and $A^H$ is the Hermitian transpose of $A$.
	
	(ii)
	As an example for psd-stability, the polynomial $f(Z) = \det Z$ on the set of symmetric
       $n \times n$-matrices is psd-stable. We postpone the proof to 
	Example~\ref{ex:ex2} below.
\end{example}

\begin{example}\label{ex:ex3}
For the polynomial 
\begin{eqnarray*}
  f(z_1,z_2,z_3)  & = & \det \left( \begin{pmatrix} 1 & 0 \\ 0 & 1 \end{pmatrix} z_1 + 
              \begin{pmatrix} 0 & 1 \\ 1 & 0 \end{pmatrix} z_2 + 
              \begin{pmatrix} 1 & 0 \\ 0 & 1 \end{pmatrix} z_3 \right) \\
  & = & (z_1 + z_3)^2 - z_2^2 \\
  & = & (z_1 + z_3 - z_2)(z_1+z_3 +z_2) \, ,
\end{eqnarray*}
Section~\ref{se:imag-proj} implies that the imaginary projection is
\[
  \mathcal{I}(f) \ = \ \{\mathbf{y} \in \R^3 \, : \, y_1 - y_2 + y_3 = 0\}
    \cup \{ \mathbf{y} \in \R^3 \, : \, y_1 + y_2 + y_3 = 0\} \, .
\]
Since, for example, $(\frac{1}{2},1,\frac{1}{2}) \in \mathcal{I}(f) \cap \R_{>0}^3$,
$f$ is not stable.
In contrast to this, setting $Z = \begin{pmatrix}z_1 & z_2 \\
  z_2 & z_3 \end{pmatrix}$, the polynomial $f(Z) = f(z_1, z_2, z_3)$
is psd-stable. Namely, for $\mathbf{y} \in \mathcal{I}(f)$,
we have
\[
  \det
  \begin{pmatrix}
    y_1 & \pm (y_1 + y_3) \\
    \pm (y_1 + y_3) & y_3
  \end{pmatrix}
  \ = \ y_1 y_3 - (y_1 + y_3)^2
  \ \le \ 0
\]
as a consequence of the arithmetic-geometric mean inequality, hence
$\mathbf{y} \not\in \inter \sym_2^+$.
\end{example}

The following lemma allows to reduce multivariate $K$-stability to univariate stable polynomials.
\begin{lemma}\label{lemma:K-stability-univariate}
A polynomial $f \in \C[\mathbf{z}] \setminus \{0\}$ 
is $K$-stable if and only if for all $\mathbf{x},\mathbf{y}\in \R^n$ with 
$\mathbf{y}\in \inter K$ the univariate polynomial $t\mapsto f(\mathbf{x}+t\mathbf{y})$ is stable.
\end{lemma}
\begin{proof}
If $f$ is not $K$-stable, then there exists $\mathbf{x} \in \R^n$ and 
$\mathbf{y} \in \inter K$ with
$f(\mathbf{x} + i \mathbf{y}) = 0$.
Hence, $i$ is a zero of the univariate polynomial $t \mapsto f(\mathbf{x} + t \mathbf{y})$
and thus that univariate polynomial is not stable.
 
Conversely, if $t \mapsto f(\mathbf{x}+t\mathbf{y})$ is not stable for $\mathbf{y}\in \inter K$,
 then there is some $\alpha+i\beta\in\C$ with $\beta>0$ and $0=f(\mathbf{x}+(\alpha+i\beta)\mathbf{y})=f(\mathbf{x}+\alpha\mathbf{y}+i\beta\mathbf{y})$. Since $\beta\mathbf{y}\in \inter K$, 
$f$ is not $K$-stable.
\end{proof}

As reviewed in Section~\ref{se:imag-proj},
for a homogeneous polynomial $f$, every component in the complement of the 
imaginary projection $\mathcal{I}(f)$ is a hyperbolicity cone. 
In particular, $f$ is stable if and only if $f$ is hyperbolic with respect to every point in the 
positive orthant \cite{garding-59}. This generalizes as follows. 

\begin{thm}\label{thm:K-stable-hyperbolic}
	Let $f \in \C[\mathbf{z}]$ be homogeneous. 
Then the following are equivalent:
	\begin{enumerate}
		\item $f$ is $K$-stable.
		\item $\mathcal{I}(f) \cap \inter K = \emptyset$. 
		\item $f$ is hyperbolic with respect to every point in $\inter K$.
	\end{enumerate}
\end{thm}

\begin{proof}
	The equivalence $(1) \Leftrightarrow (2)$ is clear. 

 ``$(3)\Rightarrow(1)$'' 
If $f$ is not $K$-stable, then there exists $\mathbf{x} \in \R$ and $\mathbf{e} \in \inter K$ 
with $f(\mathbf{x} + i \mathbf{e}) = 0$. Hence, $i$ is a root of 
$t \mapsto f(\mathbf{x} + t \mathbf{e})$, so that $f$ is not
hyperbolic with respect to $\mathbf{e}$.
 
 ``$(1)\Rightarrow(3)$'' Assume $t\mapsto f(\mathbf{x}+t\mathbf{e})$ is not hyperbolic for $\mathbf{e}\in \inter K$. In case $f(\mathbf{e}) = 0$, the point $i\mathbf{e}$ 
is a root of the homogeneous polynomial $f$ as well,
so that $f$ is not hyperbolic then. Hence, $f(\mathbf{e}) \neq 0$ and there is $\mathbf{x} \in \R$
and $\alpha+i\beta\in\C$ with $\beta \neq 0$ and 
$f(\mathbf{x}+(\alpha+i\beta)\mathbf{e}) = 0$. We can assume that $\beta > 0$, since
$- \mathbf{x} - (\alpha + i \beta) \mathbf{e}$ is a zero of $f$, too. Hence,
$0=f(\mathbf{x}+\alpha\mathbf{e}+i\beta\mathbf{e})$ and $\beta\mathbf{e}\in \inter K$, 
so that $f$ is not $K$-stable.
\end{proof}

The following consequence of the connection between $K$-stability and the imaginary
projection explains that the convexity assumption in the stability notion
is natural.

\begin{cor}
	If $f$ is $K$-stable for a non-convex cone $K$ with non-empty, 
	connected interior, 
	then it is $\cl(\conv(\inter(K)))$-stable, where $\cl(\cdot)$ and $\conv(\cdot)$
	denote the closure and the convex hull.
\end{cor}
\begin{proof}
	If $f$ is $K$-stable, then $\mathcal{I}(f) \cap \inter K = \emptyset$, that is,
	$\inter K \subseteq \mathcal{I}(f)^\compl$.
	Since $\inter K$ is connected,
	it is contained in one of the connected components of 
	$(\cl \mathcal{I}(f))^{\compl}$. Denote this component by $C$.
	The convexity of any component in $(\cl \mathcal{I}(f))^{\compl}$  
	(see Section~\ref{se:imag-proj}) implies that for $K' := \cl(\conv(\inter(K)))$,
	we have	$\inter K' \subseteq \conv \inter K \subseteq C$. Since, $C \subseteq (\cl \mathcal{I}(f))^\compl$,
	$f$ is $K'$-stable.
\end{proof}

\begin{example} \label{ex:ex2}
	We complete Example~\ref{ex:ex1} and show that $f(A) = \det A$ on
	the space of (complex) symmetric matrices is psd-stable.
	
	 Let $B \in \sym_n$ be positive definite and consider the univariate polynomial $t\mapsto f(A+tB)$. Its roots are the eigenvalues of the symmetric matrix
$-B^{-1/2}AB^{-1/2}$, where $B^{-1/2}$ denotes the unique square root of $B^{-1}$. Hence, it is real-rooted. Thus, $f$ is hyperbolic with respect to any positive definite matrix. By Theorem \ref{thm:K-stable-hyperbolic}, $f$ is psd-stable.
\end{example}

The following fact generalizes the specialization property of stable polynomials
(see, e.g., \cite[Lemma 2.4]{wagner-2011}). We will use it for the special case $K_1=\R_{\ge 0}$.

\begin{fact}\label{fact:specialization}
	Let $K=K_1\times K_2 \subset \R^n \times \R^m$ be a cone. If $f(\mathbf{z}_1,\mathbf{z}_2)$ is $K$-stable, then $f(\mathbf{a}+i\mathbf{b},\mathbf{z}_2)$ is $K_2$-stable for any $\mathbf{a} \in \R^n$ and $\mathbf{b} \in \inter K_1$.
\end{fact}

\section{A conic generalization of the HKO Theorem\label{se:hb-and-hko}}

We show that the Theorem of Hermite-Kakeya-Obreschkoff as given in 
Proposition~\ref{thm:HKO-univariate}
can be generalized to conic stability.

For $f, g \in \R[\mathbf{z}]$, we write $f \ll_K g$ if $g + if$ is $K$-stable.
First we generalize the auxiliary result in \cite[Lemma 2.8]{borcea-braenden-leeyang1}
(see also \cite[Proposition 2.7]{wagner-2011}).

\begin{thm}\label{thm:HB-for-K-stability}
	Let $f$ and $g$ be real polynomials in $\mathbf{z} = (z_1, \ldots, z_n)$. 
Then $g+if$ is $K$-stable if and only if $g+wf \in \R[\mathbf{z},w]$ is $K'$-stable, where $K'=K\times \R_{\ge 0}$.
\end{thm}
\begin{proof}
	``$\Leftarrow$'' This follows from Fact \ref{fact:specialization}, setting $w=i$.
	
	``$\Rightarrow$'' Let $g+if$ be $K$-stable.
      By Lemma \ref{lemma:K-stability-univariate}, the univariate polynomial
	\[
	t\mapsto g(\mathbf{x}+t\mathbf{y})+if(\mathbf{x}+t\mathbf{y})
	\]
	is stable for all $\mathbf{x},\mathbf{y}\in \R^n$ with $\mathbf{y}\in \inter K$. 
	For fixed $\mathbf{x}, \mathbf{y} \in \R^n$ with $\mathbf{y} \in \inter K$,
      we write $\tilde{f}(t)=f(\mathbf{x}+t\mathbf{y})$ and $\tilde{g}(t)=g(\mathbf{x}+t\mathbf{y})$ as polynomials in $\R[t]$.
      	By the univariate Hermite-Biehler Theorem~\ref{thm:HB},
      	$\tilde{f}$ interlaces $\tilde{g}$ properly,
      	in particular, 
      	$\tilde{f}$ and $\tilde{g}$ are real stable. Let $w = \alpha + i \beta$
      	with $\alpha \in \R$ and $\beta > 0$.
    By Lemma~\ref{lemma:K-stability-univariate}, we have to show that
    the univariate polynomial
    \begin{equation}
      \label{eq:tildegf}
      t \ \mapsto \ \tilde{g}+\alpha\tilde{f}+i\beta\tilde{f} \ = \ \tilde{g}+(\alpha+i\beta)\tilde{f}
    \end{equation}
    is stable.  
	By the univariate Hermite-Kakeya-Obreschkoff Theorem~\ref{thm:HKO-univariate}, the linear 
	combination $\beta \tilde{f}+\alpha\tilde{g}$ is real stable.
	Then $W(\beta\tilde{f},\tilde{g}+\alpha\tilde{f})=\beta W(\tilde{f},\tilde{g})\leq0$
	on $\R$, and thus we can deduce 
	$\beta\tilde{f}\ll\tilde{g}+\alpha\tilde{f}$.
	Invoking again the univariate Hermite-Biehler Theorem \ref{thm:HB} shows that
	the univariate polynomial~\eqref{eq:tildegf} is stable. This completes the proof.
\end{proof}

\begin{prop}\label{prop:representation-K-stability}
For every $K$-stable polynomial $h=g+if$ with $g,f\in\R[\mathbf{z}]$ the polynomials $f$ and $g$ are $K$-stable or identically zero.
\end{prop}
\begin{proof}
	By Theorem \ref{thm:HB-for-K-stability}, a non-zero polynomial
	$g+if$ is $K$-stable if and only if $g+yf$ is $K'$-stable with $K' = K \times \R_{\ge 0}$.
 Using Hurwitz's Theorem \ref{thm:Hurwitz}, sending $y\rightarrow0$ and $y\rightarrow\infty$ respectively, gives that $g$ and $f$ are $K$-stable polynomials or identically zero.
\end{proof}

Now we show the following HKO generalization for $K$-stability.

\begin{thm}[Conic HKO Theorem]\label{thm:HKO-for-K-stability}
	Let $f,g \in \R[\mathbf{z}]$. Then $\lambda f+\mu g$ is either $K$-stable or the zero polynomial for all $\lambda, \mu \in \R$
    if and only if $f+ig$ or $g+if$ is $K$-stable or $f \equiv g \equiv 0$.
\end{thm}

\begin{proof}
``$\Leftarrow$'' Let $g+if$ be $K$-stable and let $\lambda,\mu\in\R$ (the case $f+ig$ can be treated analogously). By Proposition~\ref{prop:representation-K-stability}, we can
assume $\mu \neq 0$, and hence, by factoring $\mu$, it suffices to consider $g+\lambda f$.

By Theorem \ref{thm:HB-for-K-stability}, the polynomial $g+yf$ is $K\times \R_{\ge 0}$-stable. 
Using Fact~\ref{fact:specialization}, we set $y=\lambda+i$, which gives the $K$-stable polynomial $(g+\lambda f)+if$. 
With Proposition \ref{prop:representation-K-stability}, the $K$-stability of $g+\lambda f$ follows.

``$\Rightarrow$'' Assume that $\lambda f+\mu g$ is either $K$-stable or identically zero for all $\lambda,\mu\in\R$. Let $\mathbf{x}+i\mathbf{y}\in\C^n$ with $\mathbf{y}\in \inter K$. We write $\tilde{f}(t)=f(\mathbf{x}+t\mathbf{y})$ and $\tilde{g}(t)=g(\mathbf{x}+t\mathbf{y})$.
Due to Lemma \ref{lemma:K-stability-univariate}, the univariate
polynomial $\lambda\tilde{f}+\mu\tilde{g}$ is stable.
The univariate HKO Theorem \ref{thm:HKO-univariate} implies that 
$\tilde{f}$ and $\tilde{g}$ interlace.

First, assume that $\tilde{f}$ interlaces $\tilde{g}$ 
properly for all $\mathbf{x} + i \mathbf{y} \in \C^n$
with $\mathbf{y} \in \inter K$. 
By the Hermite-Biehler Theorem~\ref{thm:HB}, $\tilde{g}+i\tilde{f}$ is stable for all $\mathbf{x}+i\mathbf{y}\in\C^n$ with $\mathbf{y}\in \inter K$, which implies $K$-stability by Lemma \ref{lemma:K-stability-univariate}.
The case where $\tilde{g}$ interlaces $\tilde{f}$ properly for all $\mathbf{x}+i\mathbf{y}\in\C^n$ 
with $\mathbf{y}\in \inter K$ is treated analogously.

It remains the case where $\tilde{f}$ interlaces $\tilde{g}$ properly for one
$\mathbf{x}_1+i\mathbf{y}_1\in\C^n$ with $\mathbf{y}_1\in \inter K$ and 
$\tilde{g}$ interlaces $\tilde{f}$ properly for another 
$\mathbf{x}_2+i\mathbf{y}_2\in\C^n$ with $\mathbf{y}_2\in \inter K$.
For $0\leq \tau\leq 1$, we consider the homotopies
\[
 \mathbf{x}_{\tau} = \tau \mathbf{x}_1+(1-\tau)\mathbf{x}_2, \quad
 \mathbf{y}_{\tau} =\tau \mathbf{y}_1+(1-{\tau})\mathbf{y}_2.
\]
The roots of $\tilde{f}$ and $\tilde{g}$ vary continuously with ${\tau}$. Since 
$\tilde{f}$ and $\tilde{g}$ 
interlace for all $x + iy \in \C^n$ with $y \in \inter K$,
there must be some $\tau \in [0,1]$ such that the roots of 
$f(\mathbf{x}_{\tau}+t\mathbf{y}_{\tau})$ 
and the roots of $g(\mathbf{x}_{\tau}+t\mathbf{y}_{\tau})$ coincide. Hence,
there is a $c\in\R$ such that $c f (\mathbf{x}_{\tau}+t\mathbf{y}_{\tau})\equiv 
g(\mathbf{x}_{\tau}+t\mathbf{y}_{\tau})$.

Let $h=cf-g$. Then $h(\mathbf{x}_{\tau}+t\mathbf{y}_{\tau})\equiv0$, which implies in particular $h(\mathbf{x}_{\tau}+i\mathbf{y}_{\tau})=0$. Due to the initial hypothesis, the polynomial
$h=cf-g$ is either $K$-stable or identically zero. Since the point $\mathbf{x}_{\tau}+i\mathbf{y}_{\tau}\in \C^n$ is a root of the polynomial $h$ with 
$\mathbf{y}_{\tau}\in \inter K$, $h$ must be identically zero.
This implies $cf\equiv g$.
Since by assumption, $f$ and $g$ are $K$-stable, and since $K$-stable polynomials remain $K$-stable
under multiplication with a complex scalar,
$f+ig$ and $g+if$ are $K$-stable as well or $f \equiv g \equiv 0$.
\end{proof}

For $f,g\in\C[\mathbf{z}]$, we denote by $W_{\mathbf{v}}(f,g):=\partial_{\mathbf{v}}f\cdot g-f\cdot\partial_\mathbf{v}g$ the $\mathbf{v}$-Wronskian of $f$ and $g$, where $\partial_\mathbf{v}$ denotes the directional derivative with respect to $\mathbf{v}$.
This allows to give a 
generalization of \cite[Theorem~2.9]{borcea-braenden-leeyang1} (see also 
\cite[Corollary 2.10]{wagner-2011}) for polyhedral and non-polyhedral
cones in terms of the directional $\mathbf{v}$-Wronskian.

\begin{thm}\label{th:polynonpoly}
For $f,g\in\R[\mathbf{z}] \setminus \{0\}$,
the following are equivalent.
\begin{enumerate}
	\item $g+if$ is $K$-stable.
	\item $g+yf$ is $K\times\R_{\ge 0}$-stable.
	\item $\lambda g+\mu f$ is $K$-stable or the zero polynomial for all $\lambda,\mu\in\R$ and $W_\mathbf{v}(f,g)\leq 0$ on $\R^n$ for all $\mathbf{v}\in\inter K$. 
\end{enumerate}
	If $K$ is a polyhedral cone $K=\cone( \mathbf{v}^{(1)},\ldots,\mathbf{v}^{(k)})$, 
  the statements are also equivalent to
\begin{enumerate}
	\item[$(4)$] $\lambda g+\mu f$ is $K$-stable or the zero polynomial
  for all $\lambda,\mu\in\R$ and $W_{\mathbf{v}^{(j)}}(f,g)\leq 0$ on $\R^n$ for all $j=1,\ldots,k$.
\end{enumerate}
\end{thm}
\begin{proof}
``(1)$\Leftrightarrow$(2)'' follows by Theorem \ref{thm:HB-for-K-stability}.

``(1)$\Rightarrow$(3)'' The first part follows by the conic HKO Theorem \ref{thm:HKO-for-K-stability}. For the second part, let $\mathbf{x}+i\mathbf{v}\in\C^n$ with $\mathbf{v}\in \inter K$. 
By Lemma \ref{lemma:K-stability-univariate}, the univariate restriction
\[
t \mapsto g(\mathbf{x}+t \mathbf{v})+if(\mathbf{x}+t \mathbf{v})
\]
is stable. The univariate
Hermite-Biehler Theorem \ref{thm:HB} implies $f(\mathbf{x}+t\mathbf{v})\ll g(\mathbf{x}+t\mathbf{v})$, i.e.,
\[
0 \ \ge \ W(f(\mathbf{x}+t\mathbf{v}),g(\mathbf{x}+t\mathbf{v})) \ = \ 
  g(\mathbf{x} + t \mathbf{v})^2 \frac{d}{dt} \left( \frac{f(\mathbf{x} + t \mathbf{v})}{g(\mathbf{x} + t \mathbf{v})} \right)
\]
for all $t \in \R$. Now the claim follows from 
\[
W_{\mathbf{v}}(f,g)(\mathbf{x}) \ = \
W\big(f(\mathbf{x}+t\mathbf{v}),g(\mathbf{x}+t\mathbf{v})\big)|_{t=0} \ \leq \ 0 \, .
\]

``(3)$\Rightarrow$(1)'' By the conic HKO Theorem \ref{thm:HKO-for-K-stability}, $f+ig$ or $g+if$ is $K$-stable. 
And by Lemma \ref{lemma:K-stability-univariate} and the univariate Hermite-Biehler 
Theorem~\ref{thm:HB}, 
for all $\mathbf{x}+i\mathbf{v}\in\C^n$ with 
$\mathbf{v}\in \inter K$, the univariate real polynomials 
$\tilde{f}(t) = f(\mathbf{x} + t \mathbf{v})$ and
$\tilde{g}(t) = g(\mathbf{x} + t \mathbf{v})$
interlace. Moreover, the elementary rule 
$\frac{d}{dt} h(\mathbf{x} + t \mathbf{v})
= \frac{\partial}{\partial \mathbf{v}} 
    h(\mathbf{z})|_{\mathbf{z} = \mathbf{x} + t \mathbf{v}}$
gives
\[
  W(\tilde{f}(t),\tilde{g}(t)) \ = \ 
  W_{\mathbf{v}}(f(\mathbf{z}),g(\mathbf{z}))|_{\mathbf{z} = \mathbf{x} + t \mathbf{v}}
  \ \le \ 0
\]
by assumption. Hence, the univariate restrictions
\[
t \mapsto g(\mathbf{x}+t \mathbf{v})+if(\mathbf{x}+t \mathbf{v})
\]
are stable, and thus by Lemma~\ref{lemma:K-stability-univariate}, $g+if$ is $K$-stable.

``(3)$\Rightarrow$(4)'' Since $K = \cone(\mathbf{v}^{(1)}, \ldots, \mathbf{v}^{(k)})$, this
  implication follows immediately from a continuity argument.

``(4)$\Rightarrow$(1)'' Let $\mathbf{x}+i\mathbf{y}\in\C^n$ with $\mathbf{y} \in \inter K$. 
We can assume that $\mathbf{y}\in\cone(\mathbf{v}^{(1)},\ldots,\mathbf{v}^{(n)})$
with linearly independent vectors $\mathbf{v}^{(1)}, \ldots, \mathbf{v}^{(n)}$.
Let $\lambda_1, \ldots, \lambda_n \ge 0$ with $\mathbf{y} = 
  \sum_{j=1}^n \lambda_j \mathbf{v}^{(j)}$.
By the precondition, $f$ and $g$ are $K$-stable,
and, by Theorem \ref{thm:HKO-for-K-stability}, $g+if$ or $f+ig$ is $K$-stable.
By Lemma \ref{lemma:K-stability-univariate}, the 
univariate restriction to $t\mapsto\mathbf{x}+t\mathbf{y}$ is stable. 
Its Wronskian fulfils
\begin{align*}
	W\big(f(\mathbf{x}+\mathbf{y}t),g(\mathbf{x}+\mathbf{y}t)\big) 
	&=g(\mathbf{x} + t \mathbf{v})^2 \frac{d}{dt} 
        \left( \frac{f(\mathbf{x}+\mathbf{y}t)}{g(\mathbf{x}+\mathbf{y}t)} \right) \, .
\end{align*}
Expressing this via
$
  \frac{d}{dt}h(\mathbf{x} + \mathbf{y}t) = 
  \sum_{j=1}^n 
    \lambda_j \frac{\partial}{\partial \mathbf{v}^{(j)}} h(\mathbf{z}) \Big|_{\mathbf{z} = \mathbf{x} + \mathbf{y} t}
$
in terms of directional derivatives, we obtain
\begin{align*}
 W\big(f(\mathbf{x}+\mathbf{y}t),g(\mathbf{x}+\mathbf{y}t)\big)  & = 
  g(\mathbf{x} + t \mathbf{v})^2 \sum_{j=1}^n \lambda_j \frac{\partial}{\partial \mathbf{v}^{(j)}} \left( 
    \frac{f(\mathbf{z})}{g(\mathbf{z})} \right) \Big|_{\mathbf{z} = \mathbf{x} + \mathbf{y} t} \\
  & =  \sum_{j=1}^n \lambda_j W_{\mathbf{v}^{(j)}}(f,g)(\mathbf{x}+\mathbf{y}t) \ \leq 0 \ .
\end{align*}

Hence, $f(\mathbf{x}+\mathbf{y}t)\ll g(\mathbf{x}+\mathbf{y}t)$, and thus
by the Hermite-Biehler Theorem \ref{thm:HB}, 
$g(\mathbf{x}+\mathbf{y}t)+if(\mathbf{x}+\mathbf{y}t)$ is stable.
By Lemma \ref{lemma:K-stability-univariate}, $g+if$ is $K$-stable. 
\end{proof}

\section{psd-stability\label{se:psdstab}}

In this section we consider the cone $K = \sym_n^+$ of positive semidefinite
matrices. In many settings, this cone provides a natural generalization of
the non-negative cone (see, e.g., \cite{bpt-2013}).
 In Theorem~\ref{thm:psd-stability-determinantal-polynomial}, 
we provide a generalization
of a stability criterion for determinantal polynomials to the psd-stability.

Recall that for two matrices, $A=(a_{ij})_{n_1\times n_2}$ and $B=(b_{ij})_{k_1\times k_2}$,
the \emph{Kronecker product} (or \emph{tensor product}) $A\otimes B$ is the 
$n_1 k_1 \times n_2 k_2$ block matrix $C=(C_{ij})_{n_1\times n_2}$ with blocks $C_{ij}=a_{ij}B$. A generalization of the 
Kronecker product is the Khatri-Rao product, which is defined as follows.

\begin{definition}\label{def:Kathri-Rao}
	Let $A=(A_{ij})_{n_1\times n_2}$ and $B=(B_{ij})_{n_1\times n_2}$ be block matrices with $n_1\times n_2$ blocks of size $p_1\times p_2$ and $q_1\times q_2$, respectively. The \emph{Khatri-Rao} product of $A$ and $B$ is defined as
	\[
	A\ast B \ = \ (A_{ij}\otimes B_{ij})_{n_1\times n_2},
	\] 
	which is a block matrix with $n_1 \times n_2$ blocks of size $p_1q_1\times p_2q_2$.
\end{definition}
Note that in the case of $p_1=p_2=1$, the Khatri-Rao product provides a scalar multiplication of the blocks $B_{ij}$ by the scalars $a_{ij}$.
And in the case $n_1=n_2=1$, $A$ and $B$ only consist of a single block and 
$A \ast B$ gives the usual Kronecker product.

While it is classically known that the Kronecker product of positive semidefinite matrices is
positive semidefinite \cite{horn-johnson-2013}, 
the following result on the Khatri-Rao product will be relevant in our situation
\cite{liu1999}.

\begin{prop}[Liu, Theorems~5 and 6 in \cite{liu1999}]
\label{thm:Khatri-Rao}
	Let $A = (A_{ij})_{n_1 \times n_2}$ and $B = (B_{ij})_{n_1 \times n_2}$ be 
	block matrices with the same block structure $n_1 \times n_2$. 
If $A$ and $B$ are positive semidefinite, then $A\ast B$ is positive semidefinite. If $A$ is positive semidefinite with
		 positive definite blocks on the diagonal and $B$ is positive definite, then $A\ast B$ is positive definite.
\end{prop}

Note that the positive semidefiniteness of $A$ implies that its blocks satisfy
$A_{ij} = A_{ji}^H$, where $A_{ij}^{H}$ denotes the Hermitian transpose of 
$A_{ij}$.
Now we show the following generalization of Proposition~\ref{pr:detcrit}
to psd-stability.

\begin{thm}\label{thm:psd-stability-determinantal-polynomial}
	Let $A = (A_{ij})_{n \times n}$ be a block matrix with $n \times n$ blocks of
	size $d \times d$.
If $A$ is positive semidefinite and $B$ is a Hermitian $d \times d$-matrix, 
then the polynomial $f(Z)=\det(\sum_{i,j=1}^n A_{ij}z_{ij} + B)$ on the set of 
symmetric $n \times n$-matrices is psd-stable or identically zero.
\end{thm}

\begin{proof}
	We write $I_d$ for the $d\times d$ identity matrix and $\mathbbm{1}_{m_1\times m_2}$ for the all-ones matrix of size $m_1\times m_2$. 
First consider the case where $A$ is positive semidefinite with positive definite blocks
on the diagonal.

Let $X, Y \in \mathcal{S}_n$ with $Y \succ 0$.
In view of Lemma \ref{lemma:K-stability-univariate}, we have to show that the univariate
polynomial $t \mapsto f(X + t Y)$ has only real roots.

We can interpret $Y$ as a block matrix with blocks of size $1\times 1$.
Using the Khatri-Rao product, $Y\ast A$ is a block matrix whose $(i,j)$-th block is $y_{ij}A_{ij}$, and we obtain the identity
	\begin{equation}
	  \label{eq:khatrirao1}
		\sum_{i,j=1}^n y_{ij} A_{ij} \ = \ (\mathbbm{1}_{1\times n}\otimes I_d)\cdot\big(Y\ast A\big)\cdot(\mathbbm{1}_{n\times 1}\otimes I_d).
	\end{equation}
   Note that the multiplication by the matrices from left and right in~\eqref{eq:khatrirao1} provides a block-wise summation of all the blocks in $Y\ast A$. 

By Proposition~\ref{thm:Khatri-Rao}, $Y\ast A$ is positive definite. 
	Hence, for $\mathbf{v}\in\R^d\setminus\{0\}$, we have 
	\begin{align*}
		&\mathbf{v}^T\left((\mathbbm{1}_{1\times n}\otimes I_d)\cdot\big(Y\ast A\big)\cdot(\mathbbm{1}_{n\times 1}\otimes I_d)\right)\mathbf{v}\\
		= \ & \
		(\mathbf{v}^T \cdots \mathbf{v}^T)\big(Y\ast A\big)(\mathbf{v}^T \cdots \mathbf{v}^T)^T \ > \ 0.
	\end{align*}
	This implies that the Hermitian matrix $Q := \sum_{i,j=1}^n A_{ij}y_{ij}$ is positive definite.

The positive definite matrix $Q$ has a square root $Q^{1/2}$. 
Set $H = \sum_{i,j=1}^n A_{ij} x_{ij} + B$.
Then, for any real symmetric $n\times n$-matrix $X$, the univariate polynomial
\begin{eqnarray*}
  t \ \mapsto \ f(X+tY) & = & \det\Big( \sum_{i,j=1}^n A_{ij} (x_{ij} + t y_{ij}) + B \Big) \\
                        & = & \det(H + tQ) \\
                        & = & \det(Q) \det(Q^{-1/2} H Q^{-1/2}+t I_d)
\end{eqnarray*}
has only real roots, since they are the negatives of the eigenvalues of the Hermitian 
matrix $Q^{-1/2} H Q^{-1/2}$.

Now, for the general case, let $A$ be a positive semidefinite matrix. 
Let $A^{(k)}=(A_{ij}^{(k)})_{n\times n}$ be a sequence of positive semidefinite block matrices with positive definite blocks on the diagonal, which approximate $A$. 
Then the polynomials $f^{(k)}(Z)=\det(\sum_{i,j=1}^n A_{ij}^{(k)}z_{ij}+B)$ 
are psd-stable and hence
have no root in the (open) Siegel upper half-plane. Due to Hurwitz's Theorem~\ref{thm:Hurwitz}, the limit polynomial $f$ is either identically zero or also non-vanishing on the Siegel upper half-plane, i.e., it is psd-stable.
\end{proof}

For an example of Theorem~\ref{thm:psd-stability-determinantal-polynomial},
observe that choosing $A$ as the block matrix with $2 \times 2$ blocks of size 
$2 \times 2$,
\[
  A_{11} = \begin{pmatrix}
    1 & 0 \\
    0 & 1
  \end{pmatrix} \, , \quad
  A_{12} = A_{21} = \begin{pmatrix}
  0 & 1/2 \\
  1/2 & 0
  \end{pmatrix} \, , \quad
  A_{22} = \begin{pmatrix}
  1 & 0 \\
  0 & 1
  \end{pmatrix}
\]
and $B=0$ results in Example~\ref{ex:ex3}. Since $A = (A_{ij})$ has the
double eigenvalues $1/2$ and $3/2$, it is positive semidefinite, so that
Theorem~\ref{thm:psd-stability-determinantal-polynomial} implies the
psd-stability of 
\[
  f(Z) \ = \ \det(A_{11} z_{11} + 2 A_{12} z_{12} + A_{22} z_{22}) \, ,
\quad \text{where } Z =  \begin{pmatrix} z_{11} & z_{12} \\
z_{12} & z_{22} \end{pmatrix} \, .
\]

The criterion stated in Theorem \ref{thm:psd-stability-determinantal-polynomial} is sufficient, but not necessary. The following is a
counterexample.

\begin{example}
Let $Z = (z_{ij})_{2 \times 2}$ be symmetric and
\begin{eqnarray*}
f(Z) & = & \det\left(\sum_{i,j=1}^2 A_{ij} z_{ij} \right) 
 \ = \ \det\left( \begin{pmatrix} 1 & 0\\ 0 & 5 \end{pmatrix}z_{11} + 2\begin{pmatrix} 0 & 
2 \\  2 & 0 \end{pmatrix}z_{12} + \begin{pmatrix} 5 & 0\\ 0 & 1 \end{pmatrix}z_{22} \right) \, .
\end{eqnarray*}
We claim that $f$ is psd-stable.
Namely, for a real matrix
$Y = (y_{ij})_{2 \times 2} \succ 0$, 
we have 
\[
f(Y) \ = \ (y_{11}+5y_{22})(5y_{11}+y_{22})-16y_{12}^2 \ > \ 5(y_{11}^2+y_{22}^2)+\big(26-16\big)y_{11}y_{22} \  > \ 0 \, ,
\]
since $y_{11},y_{22} > 0$. Hence, $\sum_{i,j=1}^2 A_{ij}y_{ij}\succ0$, and thus, 
by Example \ref{ex:ex1} (ii), $f$ is psd-stable. However, the matrix
	\[
	A \ = \ (A_{ij}) \ = \ \begin{pmatrix}
		1 & 0 & 0 & 2\\
		0 & 5 & 2 & 0\\
		0 & 2 & 5 & 0\\
		2 & 0 & 0 & 1
	\end{pmatrix}
	\]
	is not positive semidefinite, since already the  $2\times 2$-minor with indices $(1,4)$ is negative. 
\end{example}

We note that already 
the most simple case of a $2\times 2$-matrix $Z$ and diagonal coefficient matrices
$A_{ij}$ provides nonlinear conditions as the following statement shows. 

\begin{prop}
	Let $A_{ij}=\diag(a_{1}^{(ij)},\ldots,a_{d}^{(ij)})$ be diagonal $d\times d$-matrices, $1\leq i, j\leq n$. Then the block matrix $A = (A_{ij})$ is positive semidefinite if and only if for every $k \in \{1, \ldots, d\}$ the matrix
	$(a_{k}^{(ij)})_{1 \le i,j \le n}$ is positive semidefinite.
	
	For $n=2$ and a Hermitian block matrix $A = (A_{ij})$, the criterion becomes	
	\[
		a_{k}^{(11)}, a_{k}^{(22)}\geq0, \text{ and } a_{k}^{(11)}a_{k}^{(22)}-|a_{k}^{(12)}|^2 \ \geq \ 0 \ , \quad k=1,\ldots,d.
	\]
\end{prop}

\begin{proof}
	Since the blocks of $A$ are diagonal matrices, every row of $A$ contains at most
	$n$ non-zero entries. Namely, in the $k$-th row of the $i$-th block row, these are
	the elements $a_{k}^{(ij)}$, $j=1,\ldots,n$.
	By reordering the rows and columns of $A$ using a permutation matrix $P$, the resulting matrix $P^TAP$ has block diagonal structure with blocks $A_k:=(a_{k}^{(ij)})_{i,j}$ of size $n\times n$.
	Thus, $A$ is positive semidefinite if and only if each block $A_k$ is positive semidefinite. 
		
	For $n=2$, these blocks are of size $2$. Hence, the minors of $A$ consist of factors of the form $a_{k}^{(11)}a_{k}^{(22)}-a_{k}^{(12)} a_{k}^{(21)}$ 
	together with diagonal elements.
\end{proof}

\section{Conclusion and outlook}

We have introduced the concept of conic stability for multivariate polynomials
in $\C[\mathbf{z}]$ and showed generalizations of some core results for
stable polynomials to the conic stability. These positive results also show that 
the conic generalization of the stability notion appears to be very natural and fruitful.
In particular, this raises the general question to which extent the theory of 
stable polynomials, for instance stability preserving operators,
can be generalized to the conic stability. 
With regards to the
theorems on conic stability in this article, a question is if
they can be further extended to even more general types of 
stability regions.

\medskip

\noindent
{\bf Acknowledgment.} The authors thank Bernd Sturmfels for an
inspiring conversation which became a trigger for the current
paper. Moreover, thanks to an anonymous referee for helpful
suggestions.

\bibliographystyle{plain}

\begin{thebibliography}{10}

\bibitem{anari-gharan-2017}
N.~Anari and S.~Oveis Gharan.
\newblock A generalization of permanent inequalities and applications in
  counting and optimization.
\newblock In {\em Proc.\ Symp.\ Theory of Computing, Montreal}. ACM, 2017.

\bibitem{bpt-2013}
G.~Blekherman, P.A. Parrilo, and R.R. Thomas.
\newblock {\em Semidefinite Optimization and Convex Algebraic Geometry}.
\newblock SIAM, Philadelphia, PA, 2013.

\bibitem{borcea-braenden-2008}
J.~Borcea and P.~Br\"{a}nd\'{e}n.
\newblock Applications of stable polynomials to mixed determinants: {J}ohnson's
  conjectures, unimodality, and symmetrized {F}ischer products.
\newblock {\em Duke Math. J.}, 143(2):205--223, 2008.

\bibitem{borcea-braenden-leeyang1}
J.~Borcea and P.~Br\"and\'en.
\newblock The {L}ee-{Y}ang and {P}\'olya-{S}chur programs. {I}. {L}inear
  operators preserving stability.
\newblock {\em Invent. Math.}, 177(3):541--569, 2009.

\bibitem{borcea-braenden-leeyang2}
J.~Borcea and P.~Br\"{a}nd\'{e}n.
\newblock The {L}ee-{Y}ang and {P}\'{o}lya-{S}chur programs. {II}. {T}heory of
  stable polynomials and applications.
\newblock {\em Comm. Pure Appl. Math.}, 62(12):1595--1631, 2009.

\bibitem{borcea-braenden-2010}
J.~Borcea and P.~Br{\"a}nd{\'e}n.
\newblock Multivariate {P}\'olya-{S}chur classification problems in the {W}eyl
  algebra.
\newblock {\em Proc. Lond. Math. Soc.}, 101(1):73--104, 2010.

\bibitem{braenden-2011}
P.~Br{\"a}nd{\'e}n.
\newblock Obstructions to determinantal representability.
\newblock {\em Adv.\ Math.}, 226(2):1202--1212, 2011.

\bibitem{frey-lange-2006}
G.~Frey and T.~Lange.
\newblock Varieties over special fields.
\newblock In {\em Handbook of Elliptic and Hyperelliptic Curve Cryptography},
  pages 87--113. Chapman \& Hall/CRC, Boca Raton, FL, 2006.

\bibitem{garding-59}
L.~G{\aa}rding.
\newblock {An inequality for hyperbolic polynomials}.
\newblock {\em J. Math. Mech.}, 8:957--965, 1959.

\bibitem{gurvits-2008}
L.~Gurvits.
\newblock Van der {W}aerden/{S}chrijver-{V}aliant like conjectures and stable
  (aka hyperbolic) homogeneous polynomials: one theorem for all.
\newblock {\em Electron. J. Comb.}, 15(1):R66, 2008.

\bibitem{HV2007}
J.W. Helton and V.~Vinnikov.
\newblock Linear matrix inequality representation of sets.
\newblock {\em Comm. Pure Appl. Math.}, 60(5):654–674, 2007.

\bibitem{horn-johnson-2013}
R.A. Horn and C.R. Johnson.
\newblock {\em Matrix {A}nalysis}.
\newblock Cambridge University Press, second edition, 2013.

\bibitem{ito-encyclopedic-1987}
K.~It\^o, editor.
\newblock {\em Encyclopedic {D}ictionary of {M}athematics. {V}ol. {I}}.
\newblock MIT Press, Cambridge, MA, second edition, 1987.

\bibitem{joergens-theobald-2017}
T.~J\"{o}rgens and T.~Theobald.
\newblock Hyperbolicity cones and imaginary projections.
\newblock To appear in \emph{Proc.\ Amer.\ Math.\ Soc.}, 2018.

\bibitem{TTT}
T.~J\"{o}rgens, T.~Theobald, and T.~de~Wolff.
\newblock Imaginary projections of polynomials.
\newblock Preprint, {\sf arXiv:1602.02008}, 2016.

\bibitem{kpv-2015}
M.~Kummer, D.~Plaumann, and C.~Vinzant.
\newblock Hyperbolic polynomials, interlacers, and sums of squares.
\newblock {\em Math. Program.}, 153(1, Ser. B):223--245, 2015.

\bibitem{LPR2005}
A.S. Lewis, P.A. Parrilo, and M.V. Ramana.
\newblock The {L}ax conjecture is true.
\newblock {\em Proc. Amer. Math. Soc.}, 133:2495--2499, 2005.

\bibitem{liu1999}
S.~Liu.
\newblock Matrix results on the {K}hatri-{R}ao and {T}racy-{S}ingh products.
\newblock {\em Linear Algebra Appl.}, 289(1):267 -- 277, 1999.

\bibitem{mss-interlacing1}
A.W. Marcus, D.A. Spielman, and N.~Srivastava.
\newblock Interlacing families {I}: {B}ipartite {R}amanujan graphs of all
  degrees.
\newblock {\em Ann.\ Math.}, 182(1):307--325, 2015.

\bibitem{mss-interlacing2}
A.W. Marcus, D.A. Spielman, and N.~Srivastava.
\newblock Interlacing families {II}: {M}ixed characteristic polynomials and the
  {K}adison-{S}inger problem.
\newblock {\em Ann.\ Math.}, 182(1):327--350, 2015.

\bibitem{pemantle-2012}
R.~Pemantle.
\newblock Hyperbolicity and stable polynomials in combinatorics and
  probability.
\newblock In D.~Jerison, B.~Mazur, and T.~Mrowka et~al., editors, {\em Current
  Development in Mathematics 2011}, pages 57--124. Int.\ Press, Somerville, MA,
  2012.

\bibitem{rahman-schmeisser-2002}
Q.I. Rahman and G.~Schmeisser.
\newblock {\em Analytic Theory of Polynomials}, volume~26 of {\em London Math.\
  Society Monographs}.
\newblock Clarendon Press, Oxford, 2002.

\bibitem{siegel-1939}
C.L. Siegel.
\newblock Einf\"uhrung in die {T}heorie der {M}odulfunktionen {$n$}-ten
  {G}rades.
\newblock {\em Math. Ann.}, 116:617--657, 1939.

\bibitem{straszak-vishnoi-2017}
D.~Straszak and N.K. Vishnoi.
\newblock Real stable polynomials and matroids: {O}ptimization and counting.
\newblock In {\em Proc.\ Symp.\ Theory of Computing, Montreal}. ACM, 2017.

\bibitem{van-der-geer-2008}
G.~van~der Geer.
\newblock Siegel modular forms and their applications.
\newblock In {\em The 1-2-3 of modular forms}, Universitext, pages 181--245.
  Springer, Berlin, 2008.

\bibitem{wagner-2011}
D.G. Wagner.
\newblock {Multivariate stable polynomials: theory and applications}.
\newblock {\em Bull. Amer. Math. Soc.}, 48(1):53--84, 2011.

\end{thebibliography}

\end{document}